\documentclass[10pt,a4paper, reqno]{amsart}
\usepackage{mathrsfs}

\usepackage{amsmath,amsfonts,verbatim}
\usepackage{latexsym}
\usepackage{amssymb,leftidx}
\usepackage{extarrows}
\usepackage{overpic}
\usepackage{color}
\usepackage{epsfig}
\usepackage{subfigure}
\usepackage{tikz}
\usepackage{bbm}
\usepackage{tikz, caption}
\usepackage[font=small,labelfont=bf]{caption}

\usepackage[colorlinks=true,backref=page]{hyperref}
\hypersetup{urlcolor=blue, citecolor=blue}

\usepackage{amssymb}
\usepackage{mathrsfs}
\usepackage{amscd}
\usepackage{cite}

\newcommand\R{\mathbb{R}}

\usepackage{graphicx}
\usepackage{float}

\numberwithin{equation}{section}

\newtheorem{definition}{Definition}[section]
\newtheorem{lemma}{Lemma}[section]
\newtheorem{theorem}{Theorem}[section]

\newtheorem{remark}{Remark}[section]

\begin{document}
\title[Decay-estimates]{Decay estimates for Nonlinear Schr\"odinger \\equation with the inverse-square potential}

\author{Jialu Wang}
\address{Department of Mathematics, Beijing Institute of Technology, Beijing, China, 100081;} \email{jialu\_wang@bit.edu.cn}

\author{Chengbin Xu}
\address{School of Mathematics and Statistics, Qinghai Normal University, Xining, Qinghai, China, 810008; }
\email{xcbsph@163.com}

\author{Fang Zhang}
\address{Department of Mathematics, Beijing Institute of Technology, Beijing, China, 100081;} \email{zhangfang@bit.edu.cn}

\begin{abstract}
In this paper, we study the dispersive decay estimates for solution to the $3\mathrm{D}$ energy-critical nonlinear Schr\"odinger equation with an inverse-square operator $\mathcal{L}_a$ where the operator is denoted by $\mathcal{L}_{a}:=-\Delta+\frac{a}{|x|^2}$ with the constant $a\geq0$. Inspired by the work of \cite{KMVZZ1,K}, we first establish that the solutions exhibit $\dot{H}^1(\R^3)$ uniform regularity, derive the Lorentz-Strichartz estimates, and then obtain the desired decay estimates using the bootstrap argument.
The key ingredients of our approach include the equivalence of Sobolev norms and the fractional product rule.

\end{abstract}
 \maketitle

\begin{center}
\begin{minipage}{120mm}
\end{minipage}
\end{center}
\section{Introduction}
This paper is devoted to investigating the dispersive decay estimates for the energy-critical defocusing nonlinear Schr\"odinger equation (NLS) with an inverse-square potential in three
spatial dimensions, assuming initial data in ${ \dot{H}}^{1}(\R^3)$. To be precise, we consider the defocusing nonlinear Schr\"odinger equation
\begin{align}\label{equation}
\begin{cases}
i\partial_{t}u-\mathcal{L}_{a}u=|u|^4u,\ (t, x)\in \mathbb{R}\times \mathbb{R}^{3}\\
u(0, x)=u_{0}(x)\in\dot{ H}^{1}(\R^3),
\end{cases}
\end{align}
where the inverse-square potential operator is defined by $\mathcal{L}_{a}:=-\Delta+\frac{a}{|x|^2}$ with $a>-\frac14$ being a constant. We interpret $\mathcal{L}_{a}$ as the Friedrichs extension of the
quadratic form $Q$, defined initially on $C_{c}^{\infty}(\R^3\backslash\{0\})$ via
$$Q(u):=\int_{\R^3}\Big(|\nabla u|^2+\frac{|u|^2}{|x|^2}\Big)dx.$$
The restriction $a>-\frac14$ assures that the operator $\mathcal{L}_{a}$ is positive semi-definite operator, as discussed in \cite{KMVZZ}. The sharp Hardy inequality also indicates that
\begin{equation}
Q(u)=\|\sqrt{\mathcal{L}_{a}}u\|_{L_x^2}^2\sim\|\nabla u\|_{L_x^2}^2.
\end{equation}
As a consequence, the solutions to the equation \eqref{equation} conserve both mass and energy, which are defined respectively by
\begin{equation}
M(u)=\int_{\R^3}|u(t,x)|^2dx=M(u_0),
\end{equation}
and
\begin{equation}
E(u(t))=\int_{\R^3}\frac{1}{2}|\nabla u(t,x)|^2+\frac{a}{2|x|^2}|u(t,x)|^2+\frac{1}{6}|u(t,x)|^6\,dx.
\end{equation}\vspace{0.2cm}

There are several essential reasons why we are particularly concerned about this problem. First, the inverse-square operator arises frequently in various physical contexts, appearing as a scaling limit of more complex problems. There are numerous examples of such applications discussed in the mathematical literature, spanning a wide range of areas from combustion theory and the Dirac equation with Coulomb potential to the study of perturbations of classic space-time metrics such Schwarzshild and Reissner-Nordstr\"om; see \cite{BPST,BPST1,KS,PST,PST1,VZ} and the references therein.
Additional, this problem arises as a scaling limit of problems involving regular potentials with critical decay. The potential term $\frac{a}{|x|^2}$ possesses a homogeneity degree of $-2$, matching the scaling behavior of the Laplacian and rendering the operator $\mathcal{L}_a$ scale-invariant. This is the main reasons why this particular operator $\mathcal{L}_a$ has been singled out for detailed mathematical investigation.
%this problem emerges as a scaling limit of problems involving regular potentials with critical decay. The primary mathematical significance in Schr\"odinger equations
%of the form $\frac{a}{|x|^2}$ stems from the fact that the homogeneity degree of the potential term is $-2$, which causes it to scale in the same way as the Laplacian. The appearance of $\mathcal{L}_{a}$ as a scaling limit (both microscopically and astronomically) is a
%signal of one of its unique properties: $\mathcal{L}_{a}$ is scale invariant. In particular, the potential and
%Laplacian can be regarded as equally strong at every length scale. Correspondingly, problems
%involving $\mathcal{L}_{a}$ are seldom amenable to simple perturbative arguments. This is one of the reasons
%that we (and indeed many before us) have singled out this particular operator for in-depth
%study.}
In the end, dispersive estimates are fundamental tools in the research of nonlinear dispersive equations, with the nonlinear Schr\"odinger equation serving as a typical example. In recent years, the nonlinear Schr\"odinger equation with inverse-square potential has attracted considerable attention, particularly in relation to Strichartz estimates, well-posedness theory, and scattering results. Nonetheless, to the best of our knowledge, there is still a lack of results concerning dispersive decay estimates for NLS in the presence of an inverse-square potential.\vspace{0.2cm}

The significant attention has been devoted to studying the solutions of the linear Schr\"odinger equation with potential in the past years. Specifically, when the singularity of the potential function $V(x)$ is weaker than that of the inverse-square potential at the origin, such as in the case where $V(x)$ belongs to the Kato class, the properties of the Schr\"odinger semigroup $e^{itH}$ generated by the operator $H=-\Delta+V(x)$ have been extensively analyzed, see \cite{D, S, SSS,SSS1} for further details.
It is important to emphasize that the singular inverse-square potential $\frac{a}{|x|^2}$ lies in the borderline case and does not belong to the Kato class. In this context, both the strong maximum principle and Gaussian bound of the heat kernel associated with $\mathcal{L}_a$ are not valid when $a<0$. Consequently, studying dispersive equations involving the inverse-square potential presents great challenges, see \cite{BPST,VZ}. However, for $n\geq2$, the Strichartz estimates and time-decay estimates for the dispersive equations with an inverse-square potential have been established fortunately by Burq et al. \cite{BPST,BPST1,PST,PST1}. In the study of Strichartz estimates for the propagators $e^{it(-\Delta+V)}$, the decay $V(x)\sim|x|^{-2}$ is on a borderline for the validity of Strichartz estimate, refer to \cite{D,GVV}. Subsequently, Fanelli et al. in \cite{FFFP,FFFP1} developed the time-decay estimate for the Schr\"odinger flows with the scaling-critical electromagnetic potential, which includes the inverse-square potential as a specific case.
An important observation in their work is that the standard time-decay for the Schr\"odinger equation does not hold in the range $-\frac{(n-2)^2}{4}<a<0$, although Strichartz estimates remain valid in this range.
%Moreover, it is known that the Strichartz estimates hold for any potential $V(x)\sim|x|^{-2-\epsilon}$ with $\epsilon>0$, without requiring further assumptions on the monotonicity or regularity of $V(x)$; see \cite{RS}.
More recently, for dimension $n\geq3$, Mizutani \cite{M} studied endpoint Strichartz estimates for Schr\"odinger equation with the critical inverse-square potential where $a=\frac{-(n-2)^2}{4}$.
\vspace{0.2cm}

With regard to the linear dispersive decay estimate, classically going back to the work of Laplacian operator in the setting of \eqref{equation} with $a=0$. It is well known, from \cite{C,T}, that for the linear propagator $e^{it\Delta}$ of the Schr\"odinger equation in $\R^n$, one has
\begin{align}\label{estimate:l1linnftfy}
\|e^{it\Delta}f\|_{L_x^{\infty}}\lesssim t^{-\frac{n}{2}}\|f\|_{L_x^1}.
\end{align}
Also, by the conservation of mass, we have
\begin{align}\label{estimate:l22}
\|e^{it\Delta}f\|_{L^{2}}=\|f\|_{L^2},
\end{align}
which follows from the Plancherel theorem. As a result, the dispersive decay of the linear Schr\"odinger equation exhibits
\begin{align}
\|e^{it\Delta}f\|_{L_x^{p}}\lesssim t^{-n(\frac{1}{2}-\frac{1}{p})}\|f\|_{L_x^{p'}},
\end{align}
derived by interpolating \eqref{estimate:l1linnftfy} and \eqref{estimate:l22}.
On the other hand, in the study of defocusing energy-critical NLS, numerous results concerning scattering and well-posedness have been established.
The pioneering and seminal work of Bourgain \cite{Bourgain} addressed the radial case, while Colliander et al. \cite{CKSTT} proved the global well-posedness of the NLS for general initial data in spatial dimension $n=3$. Their argument was subsequently extended in \cite{RV,V,V2}, where global well-posedness was showed for spatial dimensions $n\geq4$, and simplified proofs of these results were presented in \cite{KV,V1}. For our operator $\mathcal{L}_a$, Zhang and Zheng \cite{ZZ} proved an interaction Morawetz estimate for the defocusing NLS with the inverse square potential, and then established the scattering result when $n\geq3$. We also refer the reader to scattering theory for the focusing NLS with an inverse square potential in the energy space $H^{1}(\R^n)$ for $n\geq3$, as established in \cite{Z}, which generalized the result of \cite{KMVZ,LMM} to accommodate more general nonlinear terms but with radial initial data.
Notably, in the same setting of energy-critical nonlinear Schr\"odinger equations for \eqref{equation}, Killip et al. in \cite{KMVZZ1} showed global well-posedness and scattering for the case $a>-\frac{1}{4}+\frac{1}{25}$, overcoming the difficulty that the equation preserves scale invariance, but it lacks the spatial translation invariance. Their significant result forms a robust foundation for investigating $\dot{H}^1(\R^3)$ uniform regularity in the current work. Their result can be summarized as follows:

\begin{theorem}\cite[Theorem 1.2]{KMVZZ1}\label{scatter}
Fix $a>-\frac{1}{4}+\frac{1}{25}$. Given $u_{0}\in \dot{H}^1(\R^3)$, there is a unique global solution $u$ to \eqref{equation} satisfying
\begin{equation}\label{spacetime-bound}
\int_{\R}\int_{\R^3}|u(t,x)|^{10}\,dxdt\leq C(\|u_{0}\|_{\dot{H}_{x}^{1}}).
\end{equation}
Moreover, the solution $u$ scatters, that is, there exist unique $u_{\pm}\in \dot{H}^1(\R^3)$ such that
\begin{equation}\label{spacetime-asy}
\lim_{t\rightarrow \pm \infty}\|u(t)-e^{-it\mathcal{L}_{a}}u_{\pm}\|_{\dot{H}_{x}^1}=0.
\end{equation}
\end{theorem}

In view of scattering properties and the linear dispersive decay, it is natural to explore whether the solutions to NLS also enjoy dispersive decay. Specifically, for the equation in \eqref{equation}, dispersive decay estimate was demonstrated for initial data
with $H^{1+}(\R^3)$ regularity in \cite{MWYZ}, where Ma et al. primarily focused on the setting of hyperbolic
space $\mathbb{H}^3$ and then discussed the extension to $\R^3$. This work improved upon the previous results, which required initial data to belong to $H^3(\R^3)$
regularity, see \cite{FZ,GHS}. In a later contribution, Fan et al. in \cite{FKVZ} calculated dispersive decay for the mass-critical nonlinear Schr\"odinger equation with initial data in the scaling-critical space $L^2$. Based on the method developed in \cite{FKVZ}, Kowalski \cite{K} lowered the required regularity to the scaling-critical spaces $\dot{H}^1$ for energy-critical NLS with the spatial dimensions $n=3,4$.
Inspired by \cite{K,KMVZZ1}, we investigate the dispersive decay estimate for the energy-critical defocusing nonlinear Schr\"odinger equation with the inverse-square potential in three
spatial dimensions, assuming initial-value in ${ \dot{H}}^{1}(\R^3)$.\vspace{0.2cm}

More precisely, we acquire
\begin{theorem}\label{decay-estimate}
Let  $a\geq0$
%Let the inverse square potential hamiltonian operator is denoted by
%\begin{equation}\label{operator-inverse}
%\mathcal{L}_{a}:=-\Delta+\frac{a}{|x|^2},\quad a>2
%\end{equation}
%and
and suppose $u$ solves \eqref{equation} with initial data $u_{0}\in\dot{ H}^{1}(\R^3)\cap L^{p'}(\R^3)$. Then, there exists a constant $C(\|u_{0}\|_{\dot{H}^1}, p)$ depending on $u_{0}$, such that
\begin{equation}
\|u(t,x)\|_{L_{x}^{p}}\leq C(\|u_{0}\|_{\dot{H}^1}, p) |t|^{-3(\frac{1}{2}-\frac{1}{p})}\|u_{0}\|_{L^{p'}_x},\quad 2<p<\infty
\end{equation}
uniformly for $t\neq0$.
\end{theorem}
\begin{remark}
Although the global well-posedness and scattering theory are known to hold for $a>-\frac{1}{4}+\frac{1}{25}$ as stated in Theorem \ref{scatter}, the validity of Theorem \ref{decay-estimate} can currently only be confirmed for $a\geq0$, owing to the limitations imposed by the fractional product rule. Specifically, Theorem \ref{decay-estimate} for $p=\infty$ has been established in \cite{K} when $a=0$, which corresponds to the classical Laplacian operator. However, for $a>0$, Theorem \ref{decay-estimate} with $p=\infty$ does not hold because the endpoint Sobolev embedding is not valid in this case.
\end{remark}
To prove the dispersive estimates for nonlinear Schr\"odinger equation with the inverse-square operator, we follow the argument of Kowalski in \cite{K}. Thus,
one of our primary objective is to establish the Lorentz-Strichartz estimate for the nonlinear Schr\"odinger equation \eqref{equation}. To this end, we have to address the challenges posed by the presence of the inverse-square potential and demonstrate the $\dot{H}^1$ uniform regularity of solution to equation \eqref{equation}.
It's worth noting that the space-time bound \eqref{spacetime-bound} ensures that the solutions scatter, meaning that the solutions become asymptotic to linear solutions as $t\rightarrow \pm \infty$ in the sense of \eqref{spacetime-asy}.
This scattering property, together with the equivalence of Sobolev norms and the fractional product rule, facilitates the achievement of $\dot{H}^1$ uniform regularity using the method established in \cite{CKSTT}. On the other hand, in order to accomplish Theorem \ref{decay-estimate}, we divide the proof of Theorem \ref{decay-estimate} into two different cases: $2<p<6$ and $6\leq p<\infty$, based on the integrability of time at $t=0$. Furthermore, the bootstrap argument and Sobolev embedding play an important role in these cases. However, we are unable to obtain the case $p=\infty$ due to the failure of endpoint Sobolev embedding.\vspace{0.2cm}

The paper is organized as follows. In Section \ref{Preliminaries}, as a preliminaries, we introduce some basic notations and recall related important results, including the time-decay estimates, Sobolev space, Littlewood-Paley theory, Strichartz estimates and key properties of Lorentz space. Section \ref{Lemma}
is devoted to showing several vital lemmas that are essential for the proof of Theorem \ref{decay-estimate}, including the uniform regularity of solution to equation \eqref{equation} and Lorentz-Strichartz estimates. Finally, in Section \ref{proof}, we apply the lemmas established in Section \ref{Lemma} to prove Theorem \ref{decay-estimate}.

\section{Preliminaries}\label{Preliminaries}In this preliminary section, we first introduce the notations that will be utilized throughout this paper. Additionly, we present several essential tools that form the foundation for proving Theorem \ref{decay-estimate}, including the time decay estimates, Littlewood-Paley theorem associated with the inverse-square operator $\mathcal{L}_a$, and key properties of the Lorentz space.
\subsection{Notations}We start by introducing some notations that will be employed in
this article. Let $X$ and $Y$ denote non-negative quantities, we write $X\lesssim Y$
to indicate $X\leq CY$ for constant $C$, and when the constant $C$ depends on a parameter $b$, we write $X\lesssim_{b} Y$. Additionally, $C(B)$ is used to represent a constant that may depends on $B$. We also denote $L^{p,q}(\R^3)$ as the Lorentz space and in the case where $p = q$, $L^{p,p}(\R^3)$ coincides with the standard Lebesgue space $L^p(\R^3)$. We adopt the convention that $L^{\infty,\infty}=L^{\infty}$, and we leave $L^{\infty,q}$ undefined for $q<\infty$. We utilize the notation $L_t^q L_x^r$ to denote the space-time norm defined as
\begin{align*}
\|u\|_{L_t^q L_x^r\left(\mathbb{R} \times \mathbb{R}^3\right)}:=\left(\int_{\mathbb{R}}\big(\int_{\mathbb{R}^3}|u(t, x)|^r d x\big)^{q / r} d t\right)^{1 / q}
\end{align*}
with the usual modifications when $q$ or $r$ is equal to infinity, or when the domain $\mathbb{R} \times \mathbb{R}^3$ is replaced by a smaller region of space-time, such as $I \times \mathbb{R}^3$ where $I\subset\R$. In the special case where $q=r$, we abbreviate $L_t^q L_x^q$ as $L_{t, x}^q$. For the sake of notation simplicity, we usually omit the intervals $\mathbb{R} \times \mathbb{R}^3$ and $I \times \mathbb{R}^3$ in space-time norm.

Let $I \times \mathbb{R}^3$ be a space-time slab. Based on the Strichartz estimates, we define the $L^2$ Strichartz norm $\dot{S}^0\left(I \times \mathbb{R}^3\right)$ by
\begin{equation}\label{normStrichartz}
\|u\|_{\dot{S}^0\left(I \times \mathbb{R}^3\right)}:=\sup _{(q, r) \text { admissible }}\left(\sum_N\big\|P_N^a u\big\|_{L_t^q L_x^r(I \times \mathbb{R}^3)}^2\right)^{\frac{1}{2}},
\end{equation}
and for $k=1$ we then denote the $\dot{H}^1$ Strichartz norm $\dot{S}^1\left(I \times \mathbb{R}^3\right)$ by
\begin{equation}
\|u\|_{\dot{S}^1\left(I \times \mathbb{R}^3\right)}:=\left\|\mathcal{L}_a^{\frac{1}{2}} u\right\|_{\dot{S}^0\left(I \times \mathbb{R}^3\right)}.
\end{equation}
The inclusion of the Littlewood-Paley projections in \eqref{normStrichartz} may initially appear unconventional. However, they are essential for deriving the key $L_{t}^4L_x^{\infty}$ endpoint Strichartz estimate, as demonstrated in Lemma \ref{strichartz-norm} below.

%Note that the operator $\mathcal{L}_a$ is positive precisely for $a\geq-\frac{1}{4}$ because this is the sharp form of Hardy's inequality. Many results in this paper have clearer formulations when written in terms of the parameter
%\begin{equation}
%\sigma:=\frac{1}{2}-(\frac{1}{4}+a)^{\frac{1}{2}},
%\end{equation}
%rather than the coupling constant $a$.
\subsection{Time decay estimates}In this subsection, we record the time decay estimates for Schr\"odinger operator with the inverse-square potential which is crucial for our proof of Theorem \ref{decay-estimate}.
We present here only the particular cases that will be needed in this paper (see \cite[Theorem 1.11]{FFFP}).
\begin{lemma}[Time decay for inverse-square potential \cite{FFFP} ]\label{time-decay} If $a \geq 0$ and $n=3$, then the following estimates hold:
$$
\left\|e^{i t \mathcal{L}_a} f(\cdot)\right\|_{L^p} \leq \frac{C}{|t|^{3\left(\frac{1}{2}-\frac{1}{p}\right)}}\|f\|_{L^{p^{\prime}}}, \quad p \in[2,+\infty], \quad \frac{1}{p}+\frac{1}{p^{\prime}}=1
$$
for some constant $C=C(a, p)>0$ which does not depend on $t$ and $f$.
\end{lemma}
\begin{remark}
Specially, Fanelli, Felli, Fontelos, and Primo obtained the dispersive estimates for propagator $e^{it\mathcal{L}_a}$ with $n=2,3$ in \cite{FFFP,FFFP1}. However, Miao, Su, and Zheng \cite{MSZ} extended dispersive estimates for propagator $e^{it\mathcal{L}_a}$ by proving the $W^{s,p}$-boundedness for stationary
wave operators associated with the Schr\"odinger operator with the inverse-square potential
\begin{align*}
\mathcal{L}_a=-\Delta+\frac{a}{|x|^2},\quad a\geq-\frac{(n-2)^2}{4},
\end{align*}
in dimensions $n\geq2$, as well as dispersive estimates for the wave propagator $\frac{\sin(t\sqrt{\mathcal{L}_a})}{\mathcal{L}_a}$ in all dimensions $n\geq 2$. These results represent new contributions to this field.
\end{remark}

\subsection{Sobolev spaces} Several basic harmonic analysis tools related to the operator $\mathcal{L}_{a}$ were developed in \cite{KMVZZ}, and the purpose of this subsection is to recall the equivalence of Sobolev spaces and the fractional product rule. With this in mind, we introduce Sobolev spaces adapted to the Schr\"odinger operator with the inverse-square potential, which serve as fundamental tools in harmonic analysis.

For $1<r<\infty$, we define the homogeneous and
inhomogeneous Sobolev spaces connected with the inverse-square potential operator $\mathcal{L}_a$ as $\dot{H}_{a}^{1,r}(\R^3)$ and $H_{a}^{1,r}(\R^3)$ respectively, which are equipped with the following norms
\begin{equation}\label{define-norm}
\|f\|_{\dot{H}_{a}^{1,r}(\R^3)}=\|\sqrt{\mathcal{L}_a}f\|_{L^r(\R^3)},
\end{equation}
and
\begin{equation}\label{define-norm1}
\|f\|_{H_{a}^{1,r}(\R^3)}=\|\sqrt{1+\mathcal{L}_a}f\|_{L^r(\R^3)}.
\end{equation}
For the special case when $r=2$, we simply the notation and write $\dot{H}_{a}^{1}(\R^3)=\dot{H}_{a}^{1,2}(\R^3)$ and $H_{a}^{1}(\R^3)=H_{a}^{1,2}(\R^3)$.

Estimates for equivalence of Sobolev spaces related to the operator $\mathcal{L}_a$ were establish in \cite[Theorem 1.2]{KMVZZ}, which provide conditions under which Sobolev spaces defined
through powers of $\mathcal{L}_a$ coincide with the classical Sobolev spaces.
\begin{lemma}[Equivalence of Sobolev spaces \cite{KMVZZ}] \label{Sobolev-norm} Let $n \geq 3, a \geq-\left(\frac{n-2}{2}\right)^2$, $\sigma=\frac{1}{2}-\sqrt{\frac{1}{4}+a}$, and $0<s<2$. If $1<p<\infty$ satisfies $\frac{s+\sigma}{n}<\frac{1}{p}<\min \left\{1, \frac{n-\sigma}{n}\right\}$, then
$$
\left\|(-\Delta)^{\frac{s}{2}} f\right\|_{L^p} \lesssim_{d, p, s}\left\|\mathcal{L}_a^{\frac{s}{2}} f\right\|_{L^p}, \quad \text { for all } f \in C_c^{\infty}(\mathbb{R}^n \backslash\{0\}).
$$
If $\max \left\{\frac{s}{n}, \frac{\sigma}{n}\right\}<\frac{1}{p}<\min \left\{1, \frac{n-\sigma}{n}\right\}$, which ensures already that $1<p<\infty$, then
$$
\left\|\mathcal{L}_a^{\frac{s}{2}} f\right\|_{L^p} \lesssim d, p, s\left\|(-\Delta)^{\frac{s}{2}} f\right\|_{L^p}, \quad \text { for all } f \in C_c^{\infty}(\mathbb{R}^n \backslash\{0\}).
$$
\end{lemma}
Based on the equivalence of Sobolev spaces, we can derive the fractional product rule for inverse-square operator $\mathcal{L}_a$. The $L^p$-product rule for fractional derivatives with the Laplace operator in Euclidean spaces was initially acquired by Christ and Weinstein \cite{CW}. By combining their result with Lemma \ref{Sobolev-norm}, we can get a similar result for the operator $\mathcal{L}_a$.
The main result, Theorem \ref{decay-estimate}, is confined to the range $a\geq0$. The restriction $a\geq0$ ensures that $\sigma\leq0$, which in turn leads to the following equivalence for all $1< p < 3$
\begin{equation}
\|\nabla f\|_{L^p\left(\mathbb{R}^3\right)} \sim\left\|\sqrt{\mathcal{L}_a} f\right\|_{L^p\left(\mathbb{R}^3\right)}.
\end{equation}
For more general cases, we have
\begin{lemma}[Fractional product rule] \label{rule}
Fix $a\geq0$. Then for all $f, g \in$ $C_c^{\infty}\left(\mathbb{R}^3 \backslash\{0\}\right)$, we get
$$
\left\|\mathcal{L}_a^{\frac{s}{2}}(f g)\right\|_{L^p\left(\mathbb{R}^3\right)} \lesssim\left\|\mathcal{L}_a^{\frac{s}{2}} f\right\|_{L^{p_1}\left(\mathbb{R}^3\right)}\|g\|_{L^{p_2}\left(\mathbb{R}^3\right)}+\|f\|_{L^{q_1}\left(\mathbb{R}^3\right)}\left\|\mathcal{L}_a^{\frac{s}{2}} g\right\|_{L^{q_2}\left(\mathbb{R}^3\right)}
$$
for any exponents satisfying $1< p, p_1, q_2 <\frac{3}{s}$ and $\frac{1}{p}=\frac{1}{p_1}+\frac{1}{p_2}=\frac{1}{q_1}+\frac{1}{q_2}$.

In particular, for $s=1$, we have
$$
\left\|\sqrt{\mathcal{L}_a}(f g)\right\|_{L^p\left(\mathbb{R}^3\right)} \lesssim\left\|\sqrt{\mathcal{L}_a} f\right\|_{L^{p_1}\left(\mathbb{R}^3\right)}\|g\|_{L^{p_2}\left(\mathbb{R}^3\right)}+\|f\|_{L^{q_1}\left(\mathbb{R}^3\right)}\left\|\sqrt{\mathcal{L}_a} g\right\|_{L^{q_2}\left(\mathbb{R}^3\right)}
$$
for any exponents satisfying $1< p, p_1, q_2 <3$ and $\frac{1}{p}=\frac{1}{p_1}+\frac{1}{p_2}=\frac{1}{q_1}+\frac{1}{q_2}$.
\end{lemma}
\subsection{Littlewood-Paley theory}
In this subsection, we recall the indispensable Littlewood-Paley theory, including Bernstein estimates obtained in \cite{KMVZZ}, specifically adapted to the inverse-square operator $\mathcal{L}_a$. The main result of this subsection is vital for developing the $\dot{H}^1$ uniform regularity for Schr\"odinger equation with the operator $\mathcal{L}_a$.

Let $\phi:[0, \infty) \rightarrow[0,1]$ be a smooth function such that
$$
\phi(\lambda)=1 \text { for } 0 \leq \lambda \leq 1 \text { and } \phi(\lambda)=0 \text { for } \lambda \geq 2 \text {. }
$$
For each dyadic number $N \in 2^{\mathbb{Z}}$, we define the functions
$$
\phi_N(\lambda):=\phi(\lambda / N) \text { and } \psi_N(\lambda):=\phi_N(\lambda)-\phi_{N / 2}(\lambda) \text {. }
$$
It is evident that the collection $\left\{\psi_N(\lambda)\right\}_{N \in 2^{\mathbb{Z}}}$ forms a partition of unity for $\lambda \in(0, \infty)$. We now define the Littlewood-Paley projections as follows:
$$
P_{\leq N}^a:=\phi_N\big(\sqrt{\mathcal{L}_a}\big), \quad P_N^a:=\psi_N\big(\sqrt{\mathcal{L}_a}\big), \quad \text { and } \quad P_{>N}^a:=I-P_{\leq N}^a .
$$
So, we can proceed with the analysis based on these projections. Then we have
\begin{lemma}[Bernstein estimates]\cite[Lemma 5.1]{KMVZZ} \label{Bernstein}
Let $1<p \leq q \leq \infty$ when $a \geq 0$ and let $r_0<p \leq q<$ $r_0^{\prime}=\frac{n}{\sigma}$ when $-\left(\frac{n-2}{2}\right)^2 \leq a<0$. Then we can obtain

(1) The operators $P_{\leq N}^a, P_N^a, \tilde{P}_N^a$ and $\tilde{P}_{\leq N}^a$ are bounded on $L^p$;

(2) $P_{\leq N}^a, P_N^a, \tilde{P}_N^a$ and $\tilde{P}_{\leq N}^a$ are bounded from $L^p$ to $L^q$ with norm $O\big(N^{\frac{n}{p}-\frac{n}{q}}\big)$;

(3) $N^s\left\|P_N^a f\right\|_{L^p\left(\mathbb{R}^n\right)} \sim \|\left(\mathcal{L}_a\right)^{\frac{s}{2}} P_N^a f \|_{L^p\left(\mathbb{R}^n\right)}$ for all $f \in C_c^{\infty}\left(\mathbb{R}^n\right)$ and all $s \in \mathbb{R}$.
\end{lemma}

\subsection{Strichartz estimates}
In this subsection, we review the standard Strichartz estimates for Schr\"odinger operator with the inverse-square potential in dimension $n=3$, and discuss their application to the regularity theory for non-linear Schr\"odinger equation.
\begin{definition}[Schr\"odinger-admissible]
For the spatial dimension $n=3$, we say that a pair $2\leq q,r\leq\infty$ is Schr\"odinger-admissible if $\frac{2}{q}+\frac{3}{r}=\frac{3}{2}$. We also say that $(q,r)$ is a non-endpoint Schr\"odinger-admissible pair if $2<p,q<\infty$.
Finally, we call that $(q,r)$ is Schr\"odinger-admissible with $s$ spatial derivatives if
$\frac{2}{q}+\frac{3}{r}+s=\frac{3}{2}$.
\end{definition}
Strichartz estimates for the propagator $e^{it\mathcal{L}_a}$ in $\R^3$ were established by Burq et al. in \cite{BPST1}. Combining these with the Christ-Kiselev
lemma \cite{CK} yields the following Strichartz estimates.
\begin{theorem}[Strichartz estimate \cite{BPST1}] Fix $a>-\frac{1}{4}$. The solution $u$ to nonlinear Schr\"odinger equation with the inverse-square potential
\begin{equation}
i\partial_{t}u-\mathcal{L}_{a}u=F
\end{equation}
on an interval $I\ni t_0$ obeys
\begin{equation}
\|u\|_{L_{t}^{q} L_{x}^{r}(I\times\R^3)}\lesssim \|u(t_{0})\|_{L^2(\R^3)}+\|F\|_{L_{t}^{\bar{q}'}L_{x}^{\bar{r}'}(I\times\R^3)},
\end{equation}
whenever $\frac{2}{q}+\frac{3}{r}=\frac{2}{\bar{q}}+\frac{3}{\bar{r}}=\frac{3}{2}$, $2\leq q,\bar{q}\leq\infty$, and $q\neq \bar{q}'$.
\end{theorem}
It is well-known that Strichartz estimates, along with Theorem \ref{scatter}, indicate the analogous space-time bounds for inverse-square operator $\mathcal{L}_a$
\begin{equation}
\|\sqrt{\mathcal{L}_a} u\|_{L_t^q L_x^r\left(I \times \mathbb{R}^3\right)} \lesssim\|u\|_{\dot{S}^1\left(I \times \mathbb{R}^3\right)},
\end{equation}
where $(q,r)$ denotes the Schr\"odinger-admissible pair.

Indeed, we observe that the basic inequality
\begin{align}\label{inequalilty-ele}
\left\|\left(\sum_N\left|f_N^a\right|^2\right)^{1 / 2}\right\|_{L_t^q L_x^r\left(I \times \mathbb{R}^3\right)} \leq\left(\sum_N\left\|f_N^a\right\|_{L_t^q L_x^r\left(I \times \mathbb{R}^3\right)}^2\right)^{1 / 2}
\end{align}
holds for all $2 \leq q, r \leq \infty$ and arbitrary functions $f_N$. Specifically, \eqref{inequalilty-ele} is true for all admissible exponents $(q,r)$, then we can derive
$$
\begin{aligned}
\|u\|_{L_t^q L_x^r\left(I \times \mathbb{R}^3\right)} & \lesssim\left\|\Big(\sum_N\left|P_N^a u\right|^2\Big)^{1 / 2}\right\|_{L_t^q L_x^r\left(I \times \mathbb{R}^3\right)} \\
& \lesssim\left(\sum_N\left\|P_N^a u\right\|_{L_t^q L_x^r\left(I \times \mathbb{R}^3\right)}^2\right)^{1 / 2} \\
& \lesssim\|u\|_{\dot{S}^0\left(I \times \mathbb{R}^3\right)},
\end{aligned}
$$
where we utilize the square function estimates in the first inequality above, see \cite[Theorem 5.3]{KMVZZ}.
Finally, we can acquire
\begin{equation}\label{strichartz}
\|\sqrt{\mathcal{L}_a} u\|_{L_t^q L_x^r\left(I \times \mathbb{R}^3\right)} \lesssim\|u\|_{\dot{S}^1\left(I \times \mathbb{R}^3\right)}.
\end{equation}
\subsection{Some properties of Lorentz space}\label{Lorentz-pro}In the process of proof for Theorem \ref{decay-estimate}, it will be crucial to apply Lorentz refinements to standard inequalities and techniques. This subsection focuses on recalling the relevant properties of Lorentz spaces that will be utilized in our analysis. For a comprehensive discussion of Lorentz spaces, we refer the reader to \cite{G}.

We begin by stating the Hunt interpolation, which establishes a relationship between Lorentz spaces and Lebesgue spaces.
\begin{lemma}[Hunt's interpolation] Suppose $1\leq p_1,p_2,q_1,q_2\leq\infty$ such that $p_1\neq p_2$ and $q_1\neq q_2$. Let $T$ be a sublinear operator that satisfies
\begin{equation}
\|Tf\|_{L^{p_i}}\lesssim_{p_i,q_i}\|f\|_{L^{q_i}}
\end{equation}
for $i=1,2$. Then for all $\theta\in(0,1)$ and all $0<r\leq\infty$,
\begin{equation}
\|Tf\|_{L^{p_{\theta},r}}\lesssim_{p_{\theta},q_{\theta},r}\|f\|_{L^{q_{\theta},r}},
\end{equation}
where $\frac{1}{p_{\theta}}=\frac{\theta}{p_1}+\frac{1-\theta}{p_2}$ and $\frac{1}{q_{\theta}}=\frac{\theta}{q_1}+\frac{1-\theta}{q_2}$.
\end{lemma}
Lorentz spaces share many properties and estimates found in Lebesgue spaces $L^p$. Specifically, H\"older's inequality would be extended to Lorentz spaces in a similar manner. Furthermore, Lorentz spaces satisfy the Young-O'Neil convolution inequality (see \cite{N} and \cite{O}), of which the Hardy-Littlewood-Sobolev inequality is a particular case.

\begin{lemma}[H\"older's inequality] Given $1\leq p, p_1, p_2\leq\infty$ and $0<q, q_1, q_2\leq\infty$ such that $\frac{1}{p}=\frac{1}{p_1}+\frac{1}{p_2}$ and $\frac{1}{q}=\frac{1}{q_1}+\frac{1}{q_2}$,
\begin{equation}
\|fg\|_{L^{p,q}}\lesssim_{p_i,q_i}\|f\|_{L^{p_1,q_1}}\|g\|_{L^{p_2,q_2}}.
\end{equation}
\end{lemma}
\begin{lemma}[Young-O'Neil inequality] Given $1<p, p_1, p_2<\infty$ and $0<q, q_1, q_2\leq\infty$ such that $\frac{1}{p}+1=\frac{1}{p_1}+\frac{1}{p_2}$ and $\frac{1}{q}=\frac{1}{q_1}+\frac{1}{q_2}$,
\begin{equation}
\|f\ast g\|_{L^{p,q}}\lesssim_{p_i,q_i}\|f\|_{L^{p_1,q_1}}\|g\|_{L^{p_2,q_2}}.
\end{equation}
\end{lemma}
From Hunt interpolation, the equivalence of Sobolev norm, and the usual Sobolev embedding theorems, we can establish a similar Sobolev embedding in Lorentz spaces for the inverse-square operator $\mathcal{L}_a$.

\begin{lemma}[Sobolev embedding]\label{embedding:Lorentz} Fix $1<p<\infty$, $0<s<2$, and $0<\theta\leq\infty$ such that $\frac{1}{p}+\frac{s}{3}=\frac{1}{q}$. Then
\begin{equation}
\|f\|_{L^{p,\theta}(\R^3)}\lesssim_{p,s,\theta}\||\nabla|^{s}f\|_{L^{q,\theta}(\R^3)}\lesssim_{p,s,\theta}\|\mathcal{L}_a^{\frac{s}{2}}f\|_{L^{q,\theta}(\R^3)}.
\end{equation}
Additionally, the Lorentz spaces $L^{p,q}$ exhibit a nesting property with respect to the second index $q$. To be precise, we have the continuous embedding $L^{p,q_1}\hookrightarrow L^{p,q_2}$, which implies the inequality
\begin{equation}
\|f\|_{L^{p,q_2}}\lesssim_{p,q_1,q_2}\|f\|_{L^{p,q_1}}
\end{equation}
for all $0<q_1\leq q_2\leq\infty$.
\end{lemma}
\begin{remark}\label{Lorentz:in}
For $p\geq 1$, one can obtain $|x|^{-\frac{1}{p}}\in L^{p,\infty}(\R^3)$.
\end{remark}

\section{Some techcnical Lemmas and proofs}\label{Lemma}
The primary objective of this section is to establish several key lemmas that are essential for the proof of Theorem \ref{decay-estimate}. Each lemma will be carefully stated and accomplished by a detailed proof to ensure the rigor of the subsequent arguments.
\subsection{Uniform regularity for the operator $\mathcal{L}_a$}\label{regularity:uniform} As is well-known, the uniform regularity for non-linear Schr\"odinger equation with the Laplace operator was established in \cite{CKSTT}. In this subsection, we apply an analogous argument to the inverse-square operator $\mathcal{L}_a$.

For this purpose, we first demonstrate that the $\dot{S}^1$ norm controls the following space-time norms:
\begin{lemma}\label{strichartz-norm}
For any Schwartz function $u$ on $I \times \mathbb{R}^3$,
$$
\begin{aligned}
\|\sqrt{\mathcal{L}_a} u\|_{L_{t}^{\frac{16}{3}}L_{x}^{\frac{8}{3}}}+\|u\|_{L_t^4 L_x^{\infty}}+\left\|u\right\|_{L_t^{\frac{160}{37}} L_x^{80}} \lesssim\|u\|_{\dot{S}^1},
\end{aligned}
$$
where all space-time norms are on $I \times \mathbb{R}^3$.
\end{lemma}
\begin{proof} The proof of Lemma \ref{strichartz-norm} closely follows the structure of \cite[Lemma 3.1]{CKSTT}. For the sake of clarity and fluency, we provide a brief proof here. The estimates $\|\sqrt{\mathcal{L}_a} u\|_{L_{t}^{\frac{16}{3}}L_{x}^{\frac{8}{3}}}$ and $\left\|u\right\|_{L_t^{\frac{160}{37}} L_x^{80}}$ are derived directly from \eqref{strichartz} and standard Sobolev embedding. The only exception is the $L_t^4 L_x^{\infty}$ norm, which requires more careful treatment because endpoint Sobolev embedding does not apply in the $L_x^{\infty}$ case. To address this, we introduce the quantity
$$
c_N:=\left\|P_N^a \sqrt{\mathcal{L}_a} u\right\|_{L_t^2 L_x^6}+\left\|P_N^a \sqrt{\mathcal{L}_a} u\right\|_{L_t^{\infty} L_x^2}.
$$
By the definition of $\dot{S}^1$, we then obtain the following estimate
$$
\left(\sum_N c_N^2\right)^{1 / 2} \lesssim\|u\|_{\dot{S}^1},
$$
which completes the necessary steps for establishing the required bounds.

Moreover, applying the Bernstein's inequality (see Lemma \ref{Bernstein}) yields
$$
N^{\frac{1}{2}}\left\|P_N^a u\right\|_{L_t^2 L_x^{\infty}} \lesssim c_N \text { and } N^{-\frac{1}{2}}\left\|P_N^a u\right\|_{L_t^{\infty} L_x^{\infty}} \lesssim c_N
$$
for any dyadic frequency $N$.

Define $a_N(t):=\left\|P_N^a u(t)\right\|_{L_x^{\infty}}$, we can deduce
$$
\left(\int_I a_N(t)^2 d t\right)^{1 / 2} \lesssim N^{-\frac{1}{2}} c_N,
$$
and
$$
\sup_{t \in I} a_N(t) \lesssim N^{\frac{1}{2}} c_N.
$$
Therefore, to establish the desired result, it remains to verify that
\begin{equation}\label{estimate-Lt4}
\|u\|_{L_t^4 L_x^{\infty}}^4 \lesssim \int_I\left(\sum_N a_N(t)\right)^4 d t.
\end{equation}
From the proof of \cite[Lemma 3.1]{CKSTT}, one can see \eqref{estimate-Lt4} is valid and we omit the detailed proof here.
\end{proof}

Next, we state the standard Strichartz estimates for the Schr\"odinger equation $i u_t+\mathcal{L}_a u=f$ with the inverse-square operator $\mathcal{L}_a$. Although these estimates can also be generalized to any $\dot{S}^k$ norms for the integer $k\geq0$, our focus in this article is exclusively on the case $k=1$, as it is central to the subsequent analysis of uniform regularity.
\begin{lemma}\label{estimate-strichartz}
Let $I$ be a compact time interval. Suppose $u: I \times \mathbb{R}^3 \rightarrow \mathbb{C}$ be a Schwartz solution to the Schr\"odinger equation $i u_t+\mathcal{L}_a u=f$ with initial data $u(t_0)$. Then, for any time $t_0 \in I$, and any admissible exponents $(q_1, r_1)$, we have the estimate
$$
\|u\|_{\dot{S}^1\left(I \times \mathbb{R}^3\right)} \lesssim\left\|u(t_0)\right\|_{\dot{H}^1\left(\mathbb{R}^3\right)}+C \left\|\mathcal{L}_a^{\frac{1}{2}} f\right\|_{L_t^{q_1^{\prime}} L_x^{r_1^{\prime}}}\left(I \times \mathbb{R}^3\right),
$$
where $p^{\prime}$ denotes the dual exponent to $p$, i.e., $1 / p^{\prime}+1 / p=1$.
\end{lemma}

\begin{proof} The proof of Lemma \ref{estimate-strichartz} is inspired by the approach in \cite{CKSTT}, we will provide the proof briefly, please refer to \cite{CKSTT} for more detailed procedure. By applying the Strichartz estimates, we obtain
$$
\left\|P_N^a u\right\|_{L_t^q L_x^r\left(I \times \mathbb{R}^3\right)} \lesssim\left\|P_N^a u\left(t_0\right)\right\|_{L^2\left(\mathbb{R}^3\right)}+\left\|P_N^a f\right\|_{L_t^{q_1^{\prime}} L_x^{r_1^{\prime}}\left(I \times \mathbb{R}^3\right)}
$$
for any Schr\"odinger admissible pairs $(q, r)$, $(q_1, r_1)$. Next, we apply $\mathcal{L}_a^{\frac{1}{2}}$ to both sides of the above inequality. Moreover, the Littlewood-Paley projections $P_N^a$ and the operator $\mathcal{L}_a^{\frac{1}{2}}$ commute with $\left(i \partial_t+\mathcal{L}_a\right)$
for each $N$. Finally, squaring both sides, summing over $N$, and applying the Minkowski inequality yield the desired result.
\end{proof}
In the following, let $\mathcal{O}(Y)$ denote the expression that is schematically of the form $Y$.  Specifically, $\mathcal{O}(Y)$ refers to a finite linear combination of terms that resemble $Y$, but with some factors possibly replaced by their
complex conjugates. Then, we have the following lemma.
\begin{lemma}
For any space-time slab $I \times \mathbb{R}^3$, and consider smooth functions $v_1, \ldots, v_5$ on $I \times \mathbb{R}^3$, then we can acquire
\begin{align}\label{Quintilinear}
& \left\| \mathcal{L}_a^{\frac{1}{2}}\mathcal{O}\left(v_1 v_2 v_3 v_4 v_5\right)\right\|_{L_t^1 L_x^2} \nonumber\\
& \lesssim \sum_{\{a, b, c, d, e\}=\{1,2,3,4,5\}}\left\|v_a\right\|_{\dot{S}^1}\left\|v_b\right\|_{\dot{S}^1}\left\|v_c\right\|_{L_{t,x}^{10}}\left\|v_d\right\|_{\dot{S}^1}\left\|v_e\right\|_{\dot{S}^1},
\end{align}
where all the space-time norms are on the slab $I \times \mathbb{R}^3$.
\end{lemma}

\begin{proof} Applying Lemma \ref{rule} and H\"older inequality, we encounter various terms to control including
$$
\left\|\mathcal{O}\left(\left(\mathcal{L}_a^{\frac{1}{2}} v_1\right) v_2 v_3 v_4 v_5\right)\right\|_{L_t^1 L_x^2} \lesssim\left\|\mathcal{L}_a^{\frac{1}{2}} v_1\right\|_{L_{t, }^{\frac{16}{3}}L_{x}^{\frac{8}{3}}}\left\|v_2\right\|_{L_t^4 L_x^{\infty}}\left\|v_3\right\|_{L_t^{\frac{160}{37}} L_x^{80}}\left\|v_4\right\|_{L_{t, }^{\frac{160}{37}}L_{x}^{80}}\left\|v_5\right\|_{L_{t, x}^{10}}.
$$
The claim \eqref{Quintilinear} follows then by \eqref{strichartz} and Lemma \ref{strichartz-norm}.
\end{proof}

%\begin{lemma}[Stability]\cite[Theorem 2.12]{KMVZZ1}\label{Stability} Fix $a>-\frac{1}{4}+\frac{1}{25}$. Let $I$ be a compact time interval and suppose that $\tilde{u}$ is an approximate solution to the equation \eqref{equation} on $I \times \mathbb{R}^3$ in the sense that
%$$
%\left(i \partial_t-\mathcal{L}_a\right) \tilde{u}=|\tilde{u}|^4 \tilde{u}+e
%$$
%for some function e. Assume that the following bounds for $\tilde{u}$ on the time interval $I$:
%$$
%\|\tilde{u}\|_{L_t^{\infty} \dot{H}_a^1\left(I \times \mathbb{R}^3\right)} \leq E \quad \text { and } \quad\|\tilde{u}\|_{L_{t, x}^{10}\left(I \times \mathbb{R}^3\right)} \leq L,
%$$
%where $E$ and $L$ are positive constants.
%Let $t_0 \in I$ and let $u_0 \in \dot{H}_a^1\left(\mathbb{R}^3\right)$. Assume the smallness conditions
%$$
%\left\|u_0-\tilde{u}\left(t_0\right)\right\|_{\dot{H}_a^1\left(\mathbb{R}^3\right)}+\left\|\sqrt{\mathcal{L}_a} e\right\|_{\dot{N}^0(I)} \leq \varepsilon,
%$$
%where $0<\varepsilon<\varepsilon_1=\varepsilon_1(E, L)$ and $\dot{N}^0$ denotes the dual Strichartz space of $\dot{S}^0$. Then, there exists a unique strong solution $u: I \times \mathbb{R}^3 \mapsto \mathbb{C}$ to \eqref{equation} with initial data $u_0$ at time $t=t_0$ satisfying
%$$
%\begin{aligned}
%\left\|\sqrt{\mathcal{L}_a}(u-\tilde{u})\right\|_{\dot{S}^0(I)} & \leq C(E, L) \varepsilon,
%\end{aligned}
%$$
%and
%$$
%\begin{aligned}
%\left\|\sqrt{\mathcal{L}_a} u\right\|_{\dot{S}^0(I)} & \leq C(E, L) .
%\end{aligned}
%$$
%\end{lemma}
Now, using the lemmas established above, we turn our attentions to showing uniform regularity that is one of goals of this section.
\begin{lemma}[Persistence of regularity]\label{lemma:regularity}
Let $I$ be a compact time interval, and let $u$ be a finite energy solution to equation \eqref{equation} on $I \times \mathbb{R}^3$ that satisfies the bound
$$
\|u\|_{L_{t, x}^{10}\left(I \times \mathbb{R}^3\right)} \leq M.
$$
Then, for any $t_0 \in I$, if initial data $u\left(t_0\right) \in \dot{H}^1$, the following estimate holds
\begin{equation}\label{Hk-Sk}
\|u\|_{\dot{S}^1\left(I \times \mathbb{R}^3\right)} \leq C(M, E(u))\left\|u\left(t_0\right)\right\|_{\dot{H}^1},
\end{equation}
where $C(M, E(u))$ is a constant depending on the bound $M$ and the energy $E(u)$ of the solution.
\end{lemma}

\begin{remark}For the Laplace operator $\Delta$, the $\dot{S}^k$ Strichartz norms for the case $k>1$ can be obtained using the standard iteration arguments, once the $L_{t, x}^{10}$ norm of a finite energy solution is controlled. However, for the inverse-square operator $\mathcal{L}_{a}$, due to the restriction imposed by the fractional product rule, we can only establish regularity results for $k=1$.

%In particular, for the Laplace operator $\Delta$, once the $L_{t, x}^{10}$ norm of a finite energy solution is controlled, it is possible to control all the Strichartz norms in $\dot{S}^1$. Moreover, if the initial data belongs to $H^2\left(\mathbb{R}^3\right)$, we can also control the $\dot{S}^2$ norm. From this, together with standard iteration arguments, we in fact demonstrate that a Schwartz solution can be continued in time provided that the norm $L_{t, x}^{10}$ does not blow up to infinity. However, for the inverse-square operator $\mathcal{L}_{a}$, we are only able to acquire the $\dot{S}^1$ norm due to the restriction imposed by the fractional product rule.
\end{remark}
\begin{proof} The overall idea of proof is to use a combination of Strichartz estimates and the stability theory established in \cite[Theorem 2.12]{KMVZZ1} to control the $\dot{S}^1$ norm via breaking the time interval into small subintervals.

First, applying \cite[Theorem 2.12]{KMVZZ1} with $\tilde{u}:=u$ and $e:=0$, we obtain the following estimate
$$
\|u\|_{\dot{S}^1\left(I \times \mathbb{R}^3\right)} \lesssim C(M, E).
$$
Additionally, by \eqref{Quintilinear} we also gain
\begin{equation}\label{estimate-linear}
\left\|\mathcal{L}_{a}^{\frac{1}{2}} \mathcal{O}\left(u^5\right)\right\|_{L_t^1 L_x^2} \lesssim\|u\|_{L_{t,x}^{10}}\|u\|_{\dot{S}^1}^4.
\end{equation}
A vital observation here is the presence of factor $\|u\|_{L_{t,x}^{10}}$ on the right-hand side. Next, we partition the time interval $I$ into $N$ subintervals $I_j:=\left[T_j, T_j+1\right]$ such that
$$
\|u\|_{L_{t,x}^{10}\left(I_j \times \mathbb{R}^3\right)} \leq \varepsilon,
$$
where $N \approx\left(1+\frac{M}{\varepsilon}\right)^{10}$ and $\varepsilon$ will be chosen later. Applying the Lemma \ref{estimate-strichartz} and \eqref{estimate-linear} on every interval $I_j$, we can obtain
$$
\begin{aligned}
\|u\|_{\dot{S}^1\left(I_j \times \mathbb{R}^3\right)} & \leq C\left(\left\|u\left(T_j\right)\right\|_{\dot{H}^1\left(\mathbb{R}^3\right)}+\left\|\mathcal{L}_{a}^{\frac{1}{2}}\left(|u|^4 u\right)\right\|_{L_t^1 L_x^2\left(I_j \times \mathbb{R}^3\right)}\right) \\
& \leq C\left(\left\|u\left(T_j\right)\right\|_{\dot{H}^1\left(\mathbb{R}^3\right)}+\|u\|_{L_{t,x}^{10}\left(I_j \times \mathbb{R}^3\right)} \cdot\|u\|_{\dot{S}^1\left(I_j \times \mathbb{R}^3\right)}^4\right) .
\end{aligned}
$$
Choosing $\varepsilon \leq\frac{1}{2C}(C(M, E))^{-3}$, we then acquire
\begin{equation}\label{estimate-bound}
\|u\|_{\dot{S}^1\left(I_j \times \mathbb{R}^3\right)} \leq 2 C\left\|u\left(T_j\right)\right\|_{\dot{H}^1\left(\mathbb{R}^3\right)}.
\end{equation}
Now, the bound \eqref{Hk-Sk} can be achieved by summing the bounds \eqref{estimate-bound} over each subinterval.
\end{proof}
\subsection{Lorentz space-time bounds}In this subsection, we first are devoted to proving the space-time Strichartz estimates for NLS, in which the
time and space variables to be placed in Lorentz space. To achieve this, we employ several properties of Lorentz space presented in Section \ref{Lorentz-pro}. Then, based on the established Lorentz Strichartz estimates, we derive global bounds for solutions to NLS in mixed Lorentz space-time norms, where non-endpoint Schr\"odinger-admissible pairs are restricted due to the fractional product rule.

\begin{lemma}[Lorentz-Strichartz estimates]\label{Lorentz-Strichartz}
Let $2<p,q<\infty$ be Schr\"odinger-admissible pairs. Then for all $f\in L^2$ and any space-time slab $I \times \mathbb{R}^3$, the linear evolution equation satisfies
\begin{equation}\label{mass-L}
\left\|e^{i t\mathcal{L}_a} f\right\|_{L_t^{p, 2} L_x^{q, 2}(I\times \mathbb{R}^3)} \lesssim_{p, q}\|f\|_{L^2\left(\mathbb{R}^3\right)} .
\end{equation}
Moreover, for all $0<\theta \leq \infty$, $1 \leq \phi \leq \infty$, and any time-dependent interval $I(t) \subset \mathbb{R}$,
\begin{equation}\label{nonlinear-S}
\left\|\int_{I(t)} e^{i(t-s) \mathcal{L}_a} F(s, x) d s\right\|_{L_t^{p, \theta} L_x^{q, \phi}(I\times \mathbb{R}^3)} \lesssim_{p, q, \theta, \phi}\|F\|_{L_t^{p^{\prime}, \theta} L_x^{q^{\prime}, \phi}\left(\mathbb{R}\times \mathbb{R}^3\right)} .
\end{equation}
\end{lemma}

\begin{proof} We begin with linear dispersive decay. Note that
\begin{equation}\label{mass}
\|e^{it\mathcal{L}_a}f\|_{L^2}=\|f\|_{L^2},
\end{equation}
and
\begin{equation}\label{disper}
\|e^{it\mathcal{L}_a}\|_{L_x^{\infty}}\lesssim |t|^{-\frac{3}{2}}\|f\|_{L^1},
\end{equation}
which follows from Lemma \ref{time-decay}.
Applying the Hunt interpolation inequality to \eqref{mass} and \eqref{disper} yields that for $2<q<\infty$ and $0<\phi \leq \infty$,
\begin{equation}\label{linear-L}
\left\|e^{i t \mathcal{L}_a} f\right\|_{L_x^{q, \phi}} \lesssim_{ q, \phi}|t|^{-3\left(\frac{1}{2}-\frac{1}{q}\right)}\|f\|_{L_x^{q^{\prime}, \phi}} .
\end{equation}

We first prove \eqref{nonlinear-S}, and \eqref{mass-L} directly follows from \eqref{nonlinear-S} via the $TT^{\ast}$ argument. Note that $L_x^{p, \phi}$ is normable for $1<p<\infty$ and $1 \leq \phi \leq \infty$. Thus, for any time-dependent interval $I(t) \subset \mathbb{R}$ and any space-time slab $I \times \mathbb{R}^3$, applying the triangle inequality, \eqref{linear-L}, the Young-O'Neil convolutional inequality with Remark \ref{Lorentz:in}, we can achieve
$$
\begin{aligned}
\left\|\int_{I(t)} e^{i(t-s)\mathcal{L}_a} F(s,x) d s\right\|_{L_t^{p,\theta} L_x^{q, \phi}(I\times \mathbb{R}^3)} & \lesssim_{q,\phi}\left\|\int\big\| e^{i(t-s) \mathcal{L}_a} F(s, x)\big\|_{L_x^{q, \phi}} d s\right\|_{L_t^{p, \theta}(I)} \\
& \lesssim_{q, \phi}
\left\|\int|t-s|^{-3\left(\frac{1}{2}-\frac{1}{q}\right)}\big\|F(s, x)\big\|_{L_x^{q^{\prime},\phi}} d s \right\|_{L_t^{p,\theta}(I)}\\
& \lesssim_{ p, q, \theta, \phi}\left\||t|^{-3(\frac{1}{2}-\frac{1}{q})}\right\|_{L_t^{\frac{2 q}{3(q-2)}, \infty}}\| F \|_{L_t^{p^{\prime}, \theta} L_x^{q^{\prime}, \phi}(I\times \mathbb{R}^3)}\\
&\lesssim \| F \|_{L_t^{p^{\prime}, \theta} L_x^{q^{\prime}, \phi}(I\times \mathbb{R}^3)},
\end{aligned}
$$
which concludes the proof of \eqref{nonlinear-S}.

Now, we focus on \eqref{mass-L} and proceed with a $TT^{\ast}$ argument. Define the operator $T: L_x^2 \rightarrow L_t^{p, 2} L_x^{q, 2}$ by
$$
[T f](t, x)=\left[e^{i t \mathcal{L}_a} f\right](x).
$$
Then the action of $T T^{\ast}: L_t^{p^{\prime}, 2} L_x^{q^{\prime}, 2} \rightarrow L_t^{p, 2} L_x^{q, 2}$ is given by
$$
\left[T T^{\ast} F\right](t, x)=\int e^{i(t-s) \mathcal{L}_a} F(s, x)\, ds.
$$
Applying \eqref{nonlinear-S} with $I(t)=\mathbb{R}$ and $\phi=\theta=2$ to the $T T^{\ast}$ action, we establish that the operator $T T^{\ast}$ is bounded from $L_t^{p^{\prime}, 2} L_x^{q^{\prime}, 2}$ to $L_t^{p, 2} L_x^{q, 2}$. This result indicates that the operator $T$ is bounded from $L_x^2$ to $L_t^{p, 2} L_x^{q, 2}$, thereby completing the proof of the lemma.
\end{proof}

\begin{lemma}[Space-time bounds for the Lorentz space]\label{bound-stl}
Let $\theta, \phi\geq2$ and $2<p<\infty$, $2<q<3$ be a Schr\"odinger-admissible pair. Suppose that $u_{0}\in\dot{H}^1$ obeys the hypotheses of Theorem \ref{decay-estimate}. Then the global solution $u(t)$ to \eqref{equation} with initial data $u_{0}$ satisfies
\begin{equation}\label{bound-stl1}
\|\sqrt{\mathcal{L}_a}u\|_{L_{t}^{p,\theta}L_{x}^{q,\phi}}\leq C(\|u_{0}\|_{\dot{H}^1}).
\end{equation}
\end{lemma}
\begin{proof}
From \eqref{strichartz} and Lemma \ref{lemma:regularity}, we obtain
\begin{equation}\label{bound:stl}
\|\sqrt{\mathcal{L}_a} u\|_{L_t^p L_x^q} \leq C\left(\left\|u_0\right\|_{\dot{H}^1}\right).
\end{equation}
We then turn our attention to \eqref{bound-stl1} in the Lorentz case. For all $\theta, \phi \geq 2$, utilizing the Duhamel formula
\begin{equation}\label{Duhamel}
u(t)=e^{it\mathcal{L}_a} u_0 +i \int_0^t e^{i(t-s) \mathcal{L}_a}\left[|u|^{4} u\right](s)\, ds,
\end{equation}
and Lemma \ref{Lorentz-Strichartz}, we can attain
$$
\|\sqrt{\mathcal{L}_a} u\|_{L_t^{p, \theta} L_x^{q, \phi}} \lesssim\left\|\sqrt{\mathcal{L}_a} u_0\right\|_{L_x^2}+\left\|\sqrt{\mathcal{L}_a}|u|^{4} u\right\|_{L_t^{p^{\prime}, \theta} L_x^{q^{\prime}, \phi}}.
$$
Since $p, q>2$, we observe that $p^{\prime}<\theta$ and $q^{\prime}<\phi$ for an arbitrary non-endpoint Schr\"odinger-admissible pair $(p, q)$. By Lemma \ref{rule} and the nesting property of Lorentz spaces in Lemma \ref{embedding:Lorentz}, we then estimate
$$
\begin{aligned}
\|\sqrt{\mathcal{L}_a} u\|_{L_t^{p, \theta} L_x^{q, \phi}} & \lesssim\left\|\sqrt{\mathcal{L}_a} u_0\right\|_{L_x^2}+\left\|\sqrt{\mathcal{L}_a}|u|^{4} u\right\|_{L_t^{p^{\prime}} L_x^{q^{\prime}}} \\
& \lesssim\left\|\sqrt{\mathcal{L}_a} u_0\right\|_{L_x^2}+\|\sqrt{\mathcal{L}_a} u\|_{L_t^p L_x^q}\|u\|^4_{L_t^{ \frac{4 p}{p-2}} L_x^{\frac{4 q}{q-2}}}\\
&\lesssim\left\|\sqrt{\mathcal{L}_a} u_0\right\|_{L_x^2}+\|\sqrt{\mathcal{L}_a} u\|_{L_t^p L_x^q}\|\sqrt{\mathcal{L}_a}u\|^4_{L_t^{ \frac{4 p}{p-2}} L_x^{\frac{6 p}{2p+2}}}.
\end{aligned}
$$
Due to the fact that $(\frac{4 p}{p-2}, \frac{4 q}{q-2})$ is a non-endpoint Schr\"odinger admissible pair with one spatial derivative, combining this with \eqref{bound:stl}, this concludes the proof of the Lemma \ref{bound-stl}. It's worth noting that the Lemma \ref{rule} is used in the second inequality above, so we need to impose the condition $q<3$ in Lemma \ref{bound-stl}.
\end{proof}

\section{The proof of Theorem \ref{decay-estimate}}\label{proof}
In this section, inspired by the method of \cite{K}, we focus on demonstrating Theorem \ref{decay-estimate} by utilizing the crucial Lemmas presented in Section \ref{Lemma}. According to the integrability of time at $t=0$, we divide the proof of Theorem \ref{decay-estimate} into two distinct cases: $2<p<6$ and $6\leq p<\infty$. We first handle the case $2<p<6$, where the decay estimates are relatively more straightforward and can be obtained via a bootstrap argument. Finally, we proceed with the more intricate case $6\leq p<\infty$.

It suffices to consider the case $t > 0$, as the case $t < 0$
will be obtained by time-reversal symmetry. Moreover, due to the density of Schwartz functions in $\dot{H}^1\cap L^{p'}$, it is enough to take account of the Schwartz solutions of the nonlinear Schr\"odinger equation.
\subsection{The case $2<p<6$}\label{integrable}
In this subsection, we prove Theorem \ref{decay-estimate} for the case $2<p<6$. This proof serves as a template for analyzing the remaining cases of Theorem \ref{decay-estimate}.
%For $2<p<6$, we observe that the part of linear dispersive decay is integrable near $t = 0$. The integrability of the linear decay leads to an argument that closely follows the proof of similar results for the energy-critical nonlinear Schr\"odinger equation, as discussed in \cite{K}.

\begin{proof}[The proof of the case $2<p<6$]
Now, we turn our attention to establishing Theorem \ref{decay-estimate} for the case $2<p<6$. The main idea is to apply a bootstrap argument.

For $0<T\leq \infty$, we define the norm
\begin{equation}\label{def-norm}
\|u\|_{X(T)}=\sup_{t\in[0,T)}|t|^{3(\frac{1}{2}-\frac{1}{p})}\|u(t)\|_{L_{x}^p}.
\end{equation}
Then, it is sufficient to show the following estimate:
\begin{equation}\label{bound-infty}
\|u\|_{X(\infty)}\leq C(\|u_{0}\|_{\dot{H}^1})\|u_{0}\|_{L^{p'}}.
\end{equation}
To prove this, we employ a bootstrap argument.

Let $\varepsilon>0$ denote a small parameter that will be chosen later, which depend only on
universal constants. By Lemma \ref{bound-stl}, we can partition the interval $[0,\infty)$ into several subintervals $I_{j}=[T_{j-1},T_{j})$ with $j=1,\ldots, J(\|u_{0}\|_{\dot{H}^1}, \varepsilon)$, such that on each interval, the following bound holds:
\begin{equation}\label{bound:spacetime}
\|u\|_{L_{t}^{\frac{8p}{6-p},4}L_{x}^{\frac{4p}{p-2}}(I_{j})}\leq\|\sqrt{\mathcal{L}_a}u\|_{L_{t}^{\frac{8p}{6-p},4}L_{x}^{\frac{12p}{7p-6}}(I_{j})}<\varepsilon.
\end{equation}
A straightforward calculation confirms that the pair $(\frac{8p}{6-p},\frac{4p}{p-2})$ is a non-endpoint Schr\"odinger-admissible pair with one spatial derivative. Notably, this relies on the fact that $q=\frac{8p}{6-p}<\infty$, which requires the condition $p<6$.

Next, we seek to establish the following bound for all $j=1,\ldots, J(\|u_{0}\|_{\dot{H}^1}, \varepsilon)$,
\begin{equation}\label{aimbound}
\|u\|_{X\left(T_j\right)} \lesssim\left\|u_0\right\|_{L^{p^{\prime}}}+C\left(\left\|u_0\right\|_{\dot{H}^1}\right)\|u\|_{X\left(T_{j-1}\right)}+\varepsilon^{4}\|u\|_{X\left(T_j\right)}.
\end{equation}
By selecting a sufficiently small $\varepsilon>0$ based on the constants in the above inequality \eqref{aimbound}, we can iterate this argument over $j=1, \ldots, J\left(\left\|u_0\right\|_{\dot{H}^1}\right)$, ultimately obtaining the desired estimate \eqref{bound-infty}. This completes the proof of Theorem \ref{decay-estimate} in the case $2<p<6$.

Therefore, we focus on the estimate \eqref{aimbound}. Fix $t \in\left[0, T_j\right)$, from the Duhamel formula \eqref{Duhamel}
and the linear dispersive decay in Lemma \ref{time-decay},
%the contribution of the linear term to $\|u(t)\|_{X\left(T_j\right)}$ is immediately seen to be acceptable:
%\begin{equation}\label{linear-bound}
%\left\|e^{i t \mathcal{L}_a} u_0\right\|_{X\left(T_j\right)} \lesssim\left\|u_0\right\|_{L_x^{p^{\prime}}}
%\end{equation}
we only need to compute the nonlinear part.
By Lemma \ref{time-decay}, H\"older's inequality and the definition in \eqref{def-norm}, we can acquire
\begin{align}\label{expression}
\left\|\int_0^t e^{i(t-s) \mathcal{L}_a}\left[|u|^4 u\right](s) d s\right\|_{L_x^p} &\lesssim\int_0^t|t-s|^{-3\left(\frac{1}{2}-\frac{1}{p}\right)}\|u^4(s)u(s)\|_{L^{p'}}ds\nonumber\\
&\lesssim \int_0^t|t-s|^{-3\left(\frac{1}{2}-\frac{1}{p}\right)}\|u(s)\|_{L^p}\|u(s)\|_{L_x^{\frac{4p}{p-2}}}^{4} d s \nonumber\\ & \lesssim \int_0^t|t-s|^{-3\left(\frac{1}{2}-\frac{1}{p}\right)}|s|^{-3\left(\frac{1}{2}-\frac{1}{p}\right)}\|u\|_{X(s)}\|u(s)\|_{L_x^{\frac{4p}{p-2}}}^{4} d s.
\end{align}
Then, we decompose the interval $[0, t)$ into two subintervals: $[0, t / 2)$ and $[t / 2, t)$. For $s \in[0, t / 2)$, we observe that $|t-s| \sim|t|$, and for $s \in[t / 2, t)$, we have $|s| \sim|t|$. This decomposition allows us to rewrite the expression \eqref{expression} as follows:
$$
\begin{aligned}
& \left\|\int_0^t e^{i(t-s) \mathcal{L}_a}\left[|u|^{4} u\right](s) d s\right\|_{L_x^p} \\
&\lesssim |t|^{-3\left(\frac{1}{2}-\frac{1}{p}\right)} \int_0^{t / 2}|s|^{-3\left(\frac{1}{2}-\frac{1}{p}\right)}\|u\|_{X(s)}\|u(s)\|_{L_x^{\frac{4p}{p-2}}}^{4} d s \\
&+|t|^{-3\left(\frac{1}{2}-\frac{1}{p}\right)} \int_{t / 2}^t|t-s|^{-3\left(\frac{1}{2}-\frac{1}{p}\right)}\|u\|_{X(s)}\|u(s)\|_{L_x^{\frac{4p}{p-2}}}^{4} d s.
\end{aligned}
$$
Since $2<p<6$, it follows that both $|s|^{-3\left(\frac{1}{2}-\frac{1}{p}\right)}$ and $|t-s|^{-3\left(\frac{1}{2}-\frac{1}{p}\right)}$ belong to the space $L_s^{\frac{2 p}{3(p-2)}, \infty}$ from Lemma \ref{Lorentz:in}. By applying the H\"older's inequality, we obtain
\begin{equation}\label{estimate:1}
\left\|\int_0^t e^{i(t-s)\mathcal{L}_a}\left[|u|^{4} u\right](s) d s\right\|_{L_x^p} \lesssim|t|^{-3\left(\frac{1}{2}-\frac{1}{p}\right)}\left\| \big\|u\big\|_{X(s)}\left\| u(s)\right\|_{L_x^{\frac{4p}{p-2}}}^{4}\right\|_{L_s^{\frac{2 p}{6-p}, 1}([0, t))}.
\end{equation}

For $t \in\left[0, T_j\right)$, we now decompose $[0, t)$ into two intervals: $[0, t) \cap\left[0, T_{j-1}\right)$ and $[0, t) \cap I_j$. Applying the bound in \eqref{bound:spacetime} and Lemma \ref{bound-stl}, we can get
\begin{align}\label{estimate:2}
& \left\|\big\|u\big\|_{X(s)}\big\| u(s)\big\|_{L_x^{\frac{4p}{p-2}}}^{4}\right\|_{L_s^{\frac{2p}{6-p}, 1}([0, t))}\nonumber\\
& \leq\|u\|_{X(T_{j-1})}\|u\|_{L_s^{\frac{8p}{6-p},4} L_x^{\frac{4p}{p-2}}([0, T_{j-1}))}^{4}+\|u\|_{X(T_j)}\|u\|_{L_s^{\frac{8p}{6-p},4} L_x^{\frac{4p}{p-2}}(I_j)}^{4} \nonumber\\
& \leq C\left(\|u_0\|_{\dot{H}^1}\right)\|u\|_{X(T_{j-1})}+\varepsilon^{4}\|u\|_{X(T_j)} .
\end{align}
Combining \eqref{estimate:1} with \eqref{estimate:2}, and taking the supremum over $t \in\left[0, T_j\right)$, we can attain
$$
\left\|\int_0^t e^{i(t-s) \mathcal{L}_a}\left[|u|^{4} u\right](s) d s\right\|_{X\left(T_j\right)} \lesssim C\left(\left\|u_0\right\|_{\dot{H}^1}\right)\|u\|_{X\left(T_{j-1}\right)}+\varepsilon^{4}\|u\|_{X\left(T_j\right)}.
$$
This, together with the linear dispersive term, establishes the bootstrap inequality \eqref{aimbound}. As a result, we achieve the proof of Theorem \ref{decay-estimate} for the case $2<p<6$.
\end{proof}

\subsection{The case $6\leq p<\infty$} We now conclude the proof of Theorem \ref{decay-estimate} by considering the case $6\leq p<\infty$. In this range of $p$, we adopt a similar structure of proof to the case $2<p<6$ discussed in Section \ref{integrable}, but it is important to emphasize that the linear dispersive decay is no longer integrable near $t = 0$ and thus a more refined argument is required.

\begin{proof}[The proof of the case $6\leq p<\infty$]
  As in the case $2<p<6$, we begin by decomposing the integral over $[0, t)$ into the interval $[0, t/2)$ and interval $[t/2, t)$. However, in contrast to the case $2<p<6$, these two intervals must now be handled separately due to the lack of integrability in the linear dispersive decay near $t = 0$ for the case $6 \leq p < \infty$. For this purpose, we refrain from employing the bootstrap argument that would produce a non-integrable term $|t|^{-3(\frac{1}{2}-\frac{1}{p})}$ on the interval $[0,\frac{t}{2})$.

First, we proceed with the interval $[0,\frac{t}{2})$ and aim to complete the following estimate
\begin{equation}\label{decay:early}
\left\|\int_{0}^{\frac{t}{2}}e^{it\mathcal{L}_a}(|u|^4u)(s)\,ds\right\|_{L^p}\leq C(\|u_{0}\|_{\dot{H}^1})|t|^{-3(\frac{1}{2}-\frac{1}{p})}\|u_{0}\|_{L^{p'}}.
\end{equation}
Similar to the previous case, we observe that $|t-s|\sim|t|$ for $s\in[0, t/2)$. Applying the linear dispersive decay from Lemma \ref{time-decay} as well as H\"older's inequality, we can have
\begin{align*}
\left\|\int_0^{t / 2} e^{i(t-s) \mathcal{L}_a}(|u|^4 u)(s)\, d s\right\|_{L^p}& \lesssim \int_0^{t / 2}|t-s|^{-3\left(\frac{1}{2}-\frac{1}{p}\right)}\left\|\left(|u|^4 u\right)(s)\right\|_{L_x^{\frac{p}{p-1}}} d s\nonumber\\
&\lesssim|t|^{-3\left(\frac{1}{2}-\frac{1}{p}\right)} \int_0^{t / 2}\|u(s)^{\frac{5p-6}{3p}}\|_{L_x^{\frac{9p}{5p-6}}}\|u(s)^{\frac{10p+6}{3p}}\|_{L_x^{\frac{9p}{4p-3}}}\,ds\nonumber\\
&=|t|^{-3\left(\frac{1}{2}-\frac{1}{p}\right)} \int_0^{t / 2}\|u(s)\|_{L_x^3}^{\frac{5 p-6}{3 p}}\|u(s)\|^{\frac{10 p+6}{3 p}}_{\frac{6(5 p+3)}{4 p-3}} d s.
\end{align*}
Then, by the case $2<p<6$, H\"older's inequality and Sobolev embedding, this yields
$$
\begin{aligned}
\|u(s)\|_{L_x^3}& \leq C\left(\left\|u_0\right\|_{\dot{H}^1}\right)|s|^{-1 / 2}\|u_0\|_{L_x^{\frac{3}{2}}}\\
& \leq C\left(\left\|u_0\right\|_{\dot{H}^1}\right)|s|^{-1 / 2}\left\|u_0\right\|_{L^{p^{\prime}}}^{\frac{3 p}{5 p-6}}\left\|u_0\right\|_{L^6}^{\frac{2 p-6}{5 p-6}} \\
& \leq C\left(\left\|u_0\right\|_{\dot{H}^1}\right)|s|^{-1 / 2}\left\|u_0\right\|_{L^{p^{\prime}}}^{\frac{3 p}{5 p-6}}.
\end{aligned}
$$
Therefore, we furthermore acquire
$$
\begin{aligned}
&\left\|\int_0^{t / 2} e^{i(t-s) \mathcal{L}_a}\left[|u|^4 u\right](s) d s\right\|_{L^p} \\
&\leq C\left(\left\|u_0\right\|_{\dot{H}^1}\right)|t|^{-3\left(\frac{1}{2}-\frac{1}{p}\right)}\left\|u_0\right\|_{L^{p^{\prime}}} \int_0^{t / 2}|s|^{-\frac{5 p-6}{6 p}}\|u(s)\|_{L_x^{\frac{6(5 p+3)}{4 p-3}}}^{\frac{10 p+6}{3 p}} d s \\
&\leq C\left(\left\|u_0\right\|_{\dot{H}^1}\right)|t|^{-3\left(\frac{1}{2}-\frac{1}{p}\right)}\left\|u_0\right\|_{L^{p^{\prime}}}\left\|\|u\|_{L_x^{\frac{6(5p+3)}{4p-3}}}^{\frac{10p+6}{3p}}\right\|_{L_s^{\frac{6p}{p+6},1}}\nonumber\\
&=C\left(\left\|u_0\right\|_{\dot{H}^1}\right)|t|^{-3\left(\frac{1}{2}-\frac{1}{p}\right)}\left\|u_0\right\|_{L^{p^{\prime}}}\|u\|_{L_s^{\frac{4(5 p+3)}{p+6}, \frac{10 p+6}{3 p}}L_x^{\frac{6(5 p+3)}{4 p-3}}}^{\frac{10 p+6}{3 p}}.
\end{aligned}
$$
Due to the fact that $\left(\frac{4(5 p+3)}{p+6}, \frac{6(5 p+3)}{4 p-3}\right)$ is a non-endpoint Schr\"odinger-admissible pair with one spatial derivative and $\frac{10 p+6}{3 p} \geq 2$, then we obtain
\begin{equation}
\|u\|_{L_s^{\frac{4(5 p+3)}{p+6}, \frac{10 p+6}{3 p}}L_x^{\frac{6(5 p+3)}{4 p-3}}}\leq \|\sqrt{\mathcal{L}_a}u\|_{L_s^{\frac{4(5 p+3)}{p+6}, \frac{10 p+6}{3 p}}L_x^{\frac{6(5 p+3)}{14 p+3}}}.
\end{equation}
Finally, utilizing Lemma \ref{bound-stl} achieves the proof of \eqref{decay:early}.

We now consider the interval $[t/2, t)$. On this interval, following a similar approach to the case $2<p<6$ in Section \ref{integrable}, we apply the bootstrap argument and utilize a Sobolev embedding prior to using the linear dispersive decay.

Similarly, for $0<T\leq \infty$, we also define the norm
\begin{equation}
\|u\|_{X(T)}=\sup_{t\in[0,T)}|t|^{3(\frac{1}{2}-\frac{1}{p})}\|u(t)\|_{L_x^p}.
\end{equation}
It then suffices to establish
\begin{equation}\label{bound-infty1}
\|u\|_{X(\infty)}\leq C(\|u_{0}\|_{\dot{H}^1})\|u_{0}\|_{L^{p'}}.
\end{equation}
To achieve this, we can decompose the interval $[0,\infty)$ into many subintervals $I_{j}=[T_{j-1},T_{j})$ where $j=1,\ldots, J(\|u_{0}\|_{\dot{H}^1}, \varepsilon)$. And on each subinterval, we have the bound
\begin{equation}\label{bound:spacetime1}
\|\sqrt{\mathcal{L}_a}u\|_{L_{t}^{\frac{8p}{p+3},2}L_{x}^{\frac{12p}{5p-3},2}(I_{j})}<\varepsilon,
\end{equation}
in which $\varepsilon>0$ is a small parameter to be chosen later, depending on the constant
$C(\|u_{0}\|_{\dot{H}^1})$. Moreover, due to the fact that the pair $(\frac{8p}{p+3},\frac{12p}{5p-3})$ is a non-endpoint Schr\"odinger-admissible pair, the bound \eqref{bound:spacetime1} is valid. For convenience, we write this pair as $(q,r)=(\frac{8p}{p+3},\frac{12p}{5p-3})$.

Next, using the standard iteration arguments, our goal is to show the following estimate for all $j=1,\ldots, J(\|u_{0}\|_{\dot{H}^1}, \varepsilon)$,
\begin{equation}\label{aimbound1}
\|u\|_{X(T_j)}\lesssim C(\|u_0\|_{\dot{H}^1})\left(\|u_0\|_{L^{p'}}+\|u\|_{X(T_{j-1})}+\varepsilon^4\|u\|_{X(T_j)}\right).
\end{equation}
Now, we turn our attention to verifying \eqref{aimbound1}. By combining the linear dispersive decay, the estimate \eqref{decay:early} with the Duhamel formula, for all $j$, we can acquire
\begin{equation}\label{estimate:bound2}
\|u\|_{X\left(T_j\right)} \lesssim C\left(\left\|u_0\right\|_{\dot{H}^1}\right)\left\|u_0\right\|_{L^{p^{\prime}}}+\left\|\int_{t / 2}^t e^{i(t-s) \mathcal{L}_a}\left[|u|^2 u\right](s) d s\right\|_{X\left(T_j\right)}.
\end{equation}
Thus, it remains to consider the contribution from the interval $[\frac{t}{2}, t)$.

We first take a Sobolev embedding and apply the linear dispersive decay to obtain that
\begin{align}\label{estimate-normLaa}
&\left\|\int_{t / 2}^t e^{i(t-s) \mathcal{L}_a}\left[|u|^{4} u\right](s) d s\right\|_{L^p} \nonumber\\
& \lesssim \int_{t / 2}^t\left\|e^{i(t-s)\mathcal{L}_a}\mathcal{L}_a^{\frac{1}{2}(1-\frac{3}{2 p})}\left[|u|^{4} u\right](s)\right\|_{L_x^{\frac{6 p }{2 p+3}}} d s \nonumber\\
& \lesssim \int_{t / 2}^t|t-s|^{\frac{3-p}{2 p}}\left\|\mathcal{L}_a^{\frac{1}{2}(1-\frac{3}{2 p})}\left[|u|^{4} u\right](s)\right\|_{L_x^{\frac{6p}{4p-3}}} d s.
\end{align}
It is worth noting that the Sobolev embedding here does not hold for $p=\infty$.
Combining H\"older's inequality, Lemma \ref{rule} with Sobolev embedding, we observe
\begin{align}\label{estimate-normLa}
\left\|\mathcal{L}_a^{\frac{1}{2}(1-\frac{3}{2 p})}\left[|u|^{4} u\right](s)\right\|_{L_x^{\frac{6p}{4p-3}}}&\leq \|u\|_{L^p}\|u^3\|_{L^{\frac{4p}{p-3}}}\|\mathcal{L}_a^{\frac{1}{2}(1-\frac{3}{2 p})}u\|_{L^{\frac{12p}{5p-9}}}\nonumber\\
&\leq \|u(s)\|_{L_x^p}\|\sqrt{\mathcal{L}_a} u(s)\|_{L_x^{r}}^{4},
\end{align}
provided that
$\|u(s)\|_{L^{\frac{12p}{p-3}}}\leq\|\sqrt{\mathcal{L}_a}u\|_{L^{\frac{12p}{5p-3}}}$
and
$\|\mathcal{L}_a^{\frac{1}{2}(1-\frac{3}{2 p})}u(s)\|_{L^{\frac{12p}{5p-9}}}\leq\|\sqrt{\mathcal{L}_a}u(s)\|_{L^{\frac{12p}{5p-3}}}.
$
Therefore, from Lemma \ref{embedding:Lorentz}, \eqref{estimate-normLaa} and \eqref{estimate-normLa}, it suffices to estimate
\begin{align}
\left\|\int_{t / 2}^t e^{i(t-s) \mathcal{L}_a}\left[|u|^{4} u\right](s) d s\right\|_{L^p} &\lesssim \int_{t / 2}^t|t-s|^{\frac{3-p}{2 p}}\|u(s)\|_{L_x^p}\|\sqrt{\mathcal{L}_a} u(s)\|_{L_x^{r, 2}}^{4} d s\nonumber\\
& \lesssim|t|^{-3\left(\frac{1}{2}-\frac{1}{p}\right)} \int_{t / 2}^t|t-s|^{\frac{3-p}{2 p}}\|u\|_{X(s)}\|\sqrt{\mathcal{L}_a} u(s)\|_{L_x^{r, 2}}^{4} d s \nonumber\\
& \sim|t|^{-3\left(\frac{1}{2}-\frac{1}{p}\right)}\left\| \|u(s)\|_{X(s)} \|\sqrt{\mathcal{L}_a} u(s)\|_{L_x^{r, 2}}^{4}\right\|_{L_s^{\frac{2p}{p+3}, 1}([\frac{t}{2}, t))},
\end{align}
where we use the H\"older's inequality and the fact $|s| \sim|t|$ for $s \in[\frac{t}{2}, t)$.

Ultimately, we take the supremum over $t \in\left[0, T_j\right)$ and decompose the interval $[t / 2, t)$ into two parts: $[t / 2, t) \cap I_j$ and $[t / 2, t) \cap\left[0, T_{j-1}\right)$. Since $(q,r)$ is a non-endpoint Schr\"odinger-admissible pair, applying Lemma \ref{bound-stl} and \eqref{bound:spacetime1}, we can gain
\begin{align}\label{estimate:aimbound}
&\Big\|\int_{t / 2}^t e^{i(t-s)\mathcal{L}_a}\left[|u|^{4} u\right](s) d s\Big\|_{X(T_j)} \nonumber\\
& \lesssim\|u\|_{X(T_{j-1})}\|\sqrt{\mathcal{L}_a} u\|_{L_t^{q,4}L_x^{r, 2}}^4+\|u\|_{X(T_j)}\|\sqrt{\mathcal{L}_a} u\|_{L_t^{q,4} L_x^{r, 2}(I_j)}^{4} \nonumber\\
& \lesssim C\left(\|u_0\|_{\dot{H}^1}\right)\left(\|u\|_{X(T_{j-1})}+\varepsilon^{4}\|u\|_{X(T_j)}\right).
\end{align}
\end{proof}

 From estimates \eqref{estimate:bound2} and \eqref{estimate:aimbound}, this yields the bootstrap argument \eqref{aimbound1}. By choosing a sufficiently small $\varepsilon>0$ relative to the constants in \eqref{aimbound1} and iterating over $j=1, \ldots, J\left(\left\|u_0\right\|_{\dot{H}^1}\right)$, we finish the proof of Theorem \ref{decay-estimate}.
\subsection*{Acknowledgements} C.B.X. was partially supported by Qinghai Natural Science Foundation (No.2024-ZJ-976) and NSFC (No.12401296).

\begin{center}

\end{center}
\end{document}